\documentclass[11pt]{amsart} 
\usepackage[margin=1in]{geometry}       
\usepackage[english]{babel}             
\usepackage{CJKutf8}                    
\usepackage[utf8]{inputenc}             
\usepackage{amsmath,amssymb,amsthm}     
\usepackage{graphicx}
\usepackage{tikz}
\usepackage[style=alphabetic]{biblatex} 
\usepackage{csquotes}                   
\usepackage{multirow}                   
\usepackage{booktabs}                   
\usepackage{hyperref}                   
\usepackage[all]{hypcap}                
\usepackage{caption}

\newtheorem{theorem}{Theorem}
\newtheorem{proposition}{Proposition}
\newtheorem{corollary}{Corollary}[theorem]

\theoremstyle{definition}
\newtheorem{example}{Example}

\DeclareMathOperator{\arf}{Arf}

\title{Self and Mixed Delta-Moves on Algebraically Split Links}
\author{Anthony Bosman$^\dagger$, Devin Garcia$^\dagger$, Justyce Goode$^\dagger$, Yamil Kas-Danouche$^\dagger$, Davielle Smith$^\dagger$}
\date{\today}
\address{Department of Mathematics, Andrews University, 4260 Administration Dr., Berrien Springs, MI 49104}
\email{bosman@andrews.edu}
\thanks{$^\dagger$ Andrews University; partially supported by National Science Foundation grant DMS-1950644.}

\usepackage[style=alphabetic]{biblatex} 
\begin{document}

\maketitle
\thispagestyle{empty}

\begin{abstract}
        A $\Delta$-move is a local move on a link diagram. The $\Delta$-Gordian distance between links measures the minimum number of $\Delta$-moves needed to move between link diagrams. A self $\Delta$-move only involves a single component of a link whereas a mixed $\Delta$-move involves multiple (2 or 3) components. We prove that two links are mixed $\Delta$-equivalent precisely when they have the same pairwise linking number; we also give a number of results on how (mixed/self) $\Delta$-moves relate to classical link invariants including the Arf invariant and crossing number. This allows us to produce a graph showing links related by a self $\Delta$-move for algebraically split links with up to 9-crossings. For these links we also introduce and calculate the $\Delta$-splitting number and mixed $\Delta$-splitting number, that is, the minimum number of $\Delta$-moves needed to separate the components of the link.
\end{abstract}

\section{Introduction} A $\Delta$-move is a local move on a knot or link as in Figure \ref{fig:delta_moves}. We call two links $\Delta$-equivalent if we can move from a diagram for one to the other with a series of delta moves and ambient isotopy. It is well-known that all knots are $\Delta$-equivalent to the unknot and that two links are delta equivalent precisely when their pairwise linking numbers are all the same  \cite{murakami1989certain}. A $\Delta$-move defined with any particular orientation induces those for every other, therefore we only consider unoriented links. Similarly, a $\Delta$-move induces its mirror, so there is a $\Delta$-pathway between $L$ and $L'$ exactly when there is a $\Delta$-pathway between $mL$ and $mL'$.

The minimal number of $\Delta$-moves needed to deform one link $L$ into another link $L'$ is called the $\Delta$-Gordian distance, denoted $d_G^\Delta(L,L')$. If the link $L$ is $\Delta$-equivalent to the trivial link, that is, if a link is algebraically split having all pairwise linking numbers equal to zero, then we let $u^\Delta(L)$, the $\Delta$-unlinking number, denote the $\Delta$-Gordian distance between $L$ and the trivial link. The $\Delta$-unknotting number and $\Delta$-Gordian distance have been well-studied for knots \cite{nakamura1998delta} with many findings extended to algebraically split links \cite{REU21}.

\begin{figure}
\begin{center}
\begin{tikzpicture}
\node[inner sep=0pt] (Phase1) at (0,0)
    {\includegraphics[width=.25\textwidth]{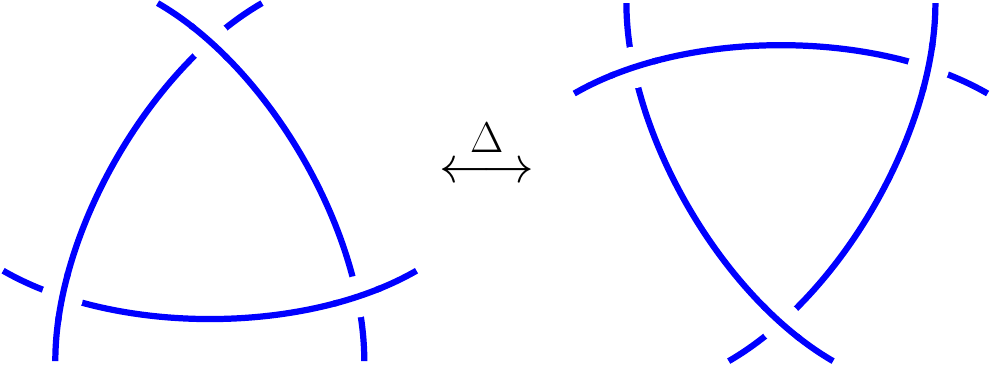}};
\node[inner sep=0pt] (Phase2) at (5,0)
    {\includegraphics[width=.25\textwidth]{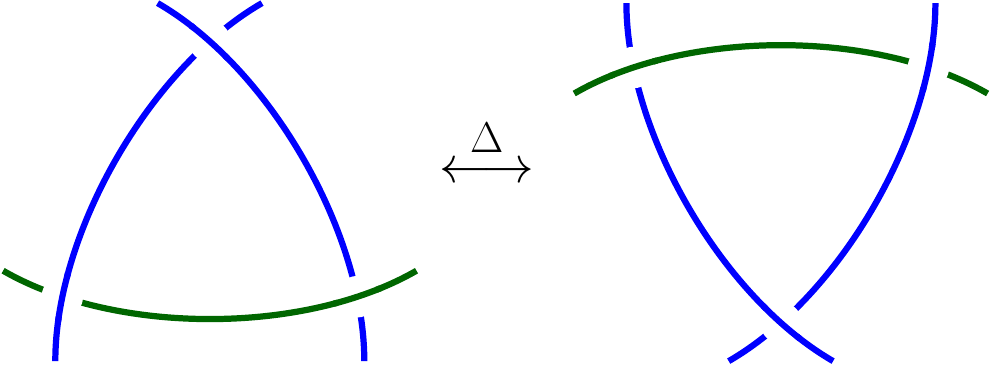}};
\node[inner sep=0pt] (Phase3) at (10,0)
    {\includegraphics[width=.25\textwidth]{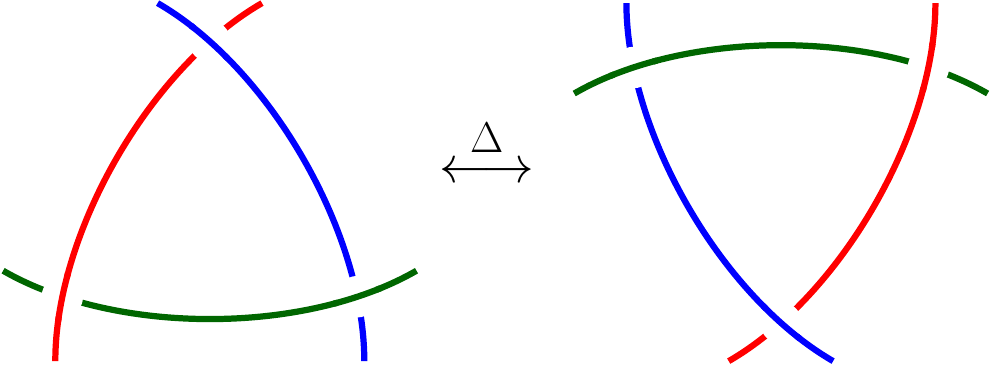}};
\end{tikzpicture}
\end{center}
\caption{Left: A self $\Delta$-move. Center and Right: Mixed $\Delta$-moves.}
    \label{fig:delta_moves}
\end{figure}

In the case of links with at least two components, we can classify $\Delta$-moves according to the number of components involved in the $\Delta$-move. In particular, if all three strands belong to the same component of the link, the move is called a self $\Delta$-move. If the strands belong to either two or three components, then we will refer to the move as a mixed $\Delta$-move. We can then introduce the notion of self or mixed $\Delta$-equivalence and self or mixed $\Delta$-Gordian distance by restricting the permitted delta moves to either only self or mixed $\Delta$-moves.

Links that are $\Delta$-equivalent have been classified using the coefficients of the Conway polynomial \cite{nakanishi02, nakanishi03}. Under this classification, we exhibit in Section 2 graphs of self $\Delta$-moves connecting two component algebraically split links with up to 9 crossings. This lets us tabulate the delta Gordian distance between such pairs of links.

In the case of mixed delta moves, we prove the following result:
\begin{theorem}
    Two links are mixed $\Delta$-equivalent if and only if they have the same components and pairwise linking numbers.
\end{theorem}

It immediately follows:
\begin{corollary}
    A link $L$ is mixed $\Delta$-equivalent to a split link if and only if $L$ is an algebraically split link.
\end{corollary}

Given this result, it is natural to consider the fewest number of mixed delta moves needed to deform an algebraically split link into a split link. We call this the mixed $\Delta$-splitting number and denote it $sp^{m\Delta}(L)$. If we instead permit any kind of delta move, then we have the $\Delta$-splitting number, denoted $sp^{\Delta}(L)$. In Section 3 we relate $sp^{m\Delta}(L)$ and $sp^{\Delta}(L)$ to the Arf invariant, $\Delta$-unknotting number, and traditional splitting number then we use these relations to determine the (mixed) $\Delta$-splitting number for algebraically split links with up to 9 crossings.

Throughout the paper we follow the Rolfsen naming convention for knots and links as used in Knot Atlas \cite{knotatlas}.

\section{Self Delta-Pathways}
The Arf invariant is a binary invariant of a knot that is changed by a $\Delta$-move on a knot. Since a self $\Delta$-move on a link is a $\Delta$-move on a single component of the link, it will change the Arf invariant of that component. Thus, a self $\Delta$-move changes the knot type of the component. A self $\Delta$-pathway between links is a series of transformations by self $\Delta$-moves and ambient isotopy. We call the number of self-$\Delta$ moves the length of the self $\Delta$-pathway.

\begin{proposition}
    If there exists two self $\Delta$-pathways of length $n_1$ and $n_2$ between the links $L$ and $L'$, then $n_1\equiv n_2\equiv\sum_{i=1}^m \arf(L_i)+\sum_{i=1}^m \arf(L_i')\pmod{2}$.
\end{proposition}

{\bf Proof:} A self $\Delta$-move on $L$ changes the parity of $\sum_{i=1}^m \arf(L_i)$. Therefore,
\[\sum_{i=1}^m \arf(L_i)+n_1\equiv\sum_{i=1}^m \arf(L_i')\pmod{2}\]
and similarly,
\[\sum_{i=1}^m \arf(L_i)+n_2\equiv\sum_{i=1}^m \arf(L_i')\pmod{2}.\]
Thus,
\[n_1\equiv\sum_{i=1}^m \arf(L_i)+\sum_{i=1}^m \arf(L_i')\equiv n_2 \pmod{2}.\]
$\square$

In particular, since the self $\Delta$-Gordian distance is defined to be the minimal self $\Delta$-pathway between two links, we have:

\begin{corollary}
Given self $\Delta$-equivalent links $L$ and $L'$,
\[d_G^{s\Delta}(L,L')\equiv \sum_{i=1}^m \arf(L_i)+\sum_{i=1}^m \arf(L_i')\pmod{2}.\]
\end{corollary}

The Arf invariant can be extended to proper links \cite{hoste1984arf,robertello1965invariant}, that is, links $L$ such that \[\sum_{1\leq i<j\leq m}lk(L_i,L_j)\equiv 0\pmod{2}.\] In particular, the Arf invariant is well defined for algebraically split links (i.e. links with vanishing pairwise linking number). It has been shown that a $\Delta$-move changes the parity of the Arf invariant of a proper link and therefore $d^{\Delta}(L,L')\equiv \arf(L)+\arf(L')\pmod{2}$ \cite{REU21}. Thus restricting to just self $\Delta$-moves, we have $d^{s\Delta}(L,L')\equiv \arf(L)+\arf(L')\pmod{2}$.\\

\begin{example}
    Since both the links $L8a2$ and $L8a4$ can be deformed into the trivial link with a single self $\Delta$-move, $1\leq d_G^{s\Delta}(L8a2,L8a4)\leq2$. Yet
    \[d_G^{s\Delta}(L8a2,L8a4)\equiv \arf(L8a2)+\arf(L8a4)\equiv 0\pmod{2}\]
and therefore $d_G^{s\Delta}(L8a2,L8a4)=2$.\\
\end{example}

Additionally, since each self $\Delta$-move only changes the knot type of a single component, we have the following:\\

\begin{proposition}
Given self $\Delta$-equivalent links $L$ and $L'$,
\[d_G^{s\Delta}(L,L')\geq\left|\sum_{i=1}^m u^\Delta(L_i)-\sum_{i=1}^m u^\Delta(L_i')\right|.\]
\end{proposition}

\begin{example}
    There is a self-$\Delta$ pathway from $mL5a1$ to $L7n2$ to $L9n3$ to $L9n6$. Since
    \[\arf(mL5a1)+\arf(L9n6)=1+0=1\]
    therefore $d^{s\Delta}(mL5a1,L9n6)$ is either 1 or 3. However, since $mL5a1$ has as its components two unknots, $L9n6$ has as its components an unknot and $5_1$, and $u^{\Delta}(5_1)=3$ \cite{okada90}, we conclude $d^{s\Delta}(mL5a1,L9n6)=3$.
\end{example}

Given a link $L$, define
\[\delta_1(L)=a_{1}(L)\text{ and}\]
\[\delta_2(L)=a_{3}(L) - a_{2}(L) \times (a_2(L_1)+a_2(L_2)) \]
where the $a_i$ are coefficients of the Conway polynomial. Nakanishi-Ohyama showed that two 2-component links are self $\Delta$-equivalent precisely when they have the same values for $\delta_1$ and $\delta_2$ \cite{nakanishi02}\cite{nakanishi03}. The sign of $\delta_i$ is changed by mirroring the link. The below graphs exhibit the self $\Delta$-pathways for prime links with up to 9 crossing in the families $(\delta_1,\delta_2)=(0,0)$, $(0,\pm1)$, and $(0,\pm2)$. Moreover, for each pair of links $L$ and $L'$, we will use the above restrictions to tabulate the possible values for $d_G^{s\Delta}(L,L')$. Note that if there exists a self $\Delta$-pathway between $L$ and $L'$, then there exists such a pathway between their mirrors $mL$ and $mL'$. Therefore $d_G^{s\Delta}(L,L')=d_G^{s\Delta}(mL,mL')$. Values with asterisks denote upper bounds to $d_G^{s\Delta}(L,L')$; the exact value is a positive integer with same parity at most the upper bound.\\


\begin{minipage}{7cm}
    \begin{tabular}{|c|c|c|c|}
     \hline
     \multicolumn{4}{|c|}{($\delta_{1}$,$\delta_{2}$)=(0,0)} \\
     \hline
     & Trivial & L8a2 & L8a4  \\
     \hline
     Trivial & 0 & 1 & 1  \\
     \hline
     L8a2 & - & 0 & 2 \\
     \hline
     L8a4 & - & - & 0 \\
     \hline
    \end{tabular}
\end{minipage}
\begin{minipage}{10cm}
    \begin{center}
    \includegraphics[width=4cm]{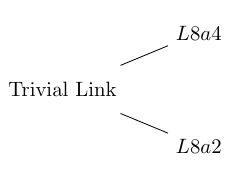}
    \captionof{figure}{($\delta_{1}$,$\delta_{2}$)=(0,0)}
    \label{fig:0_pathways}
    \end{center}
\end{minipage}

\begin{minipage}{10cm}
    \begin{center}
    \includegraphics[width=9.3cm]{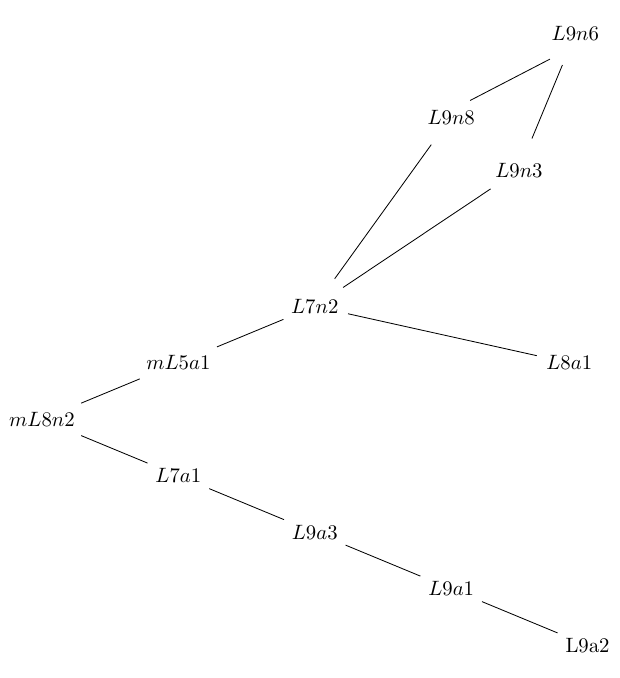}
    \captionof{figure}{($\delta_{1}$,$\delta_{2}$)=(0,$\pm$1)}
    \label{fig:1_pathways}
    \end{center}
\end{minipage}
\begin{minipage}{7cm}
    \begin{center}
    \includegraphics[width=7cm]{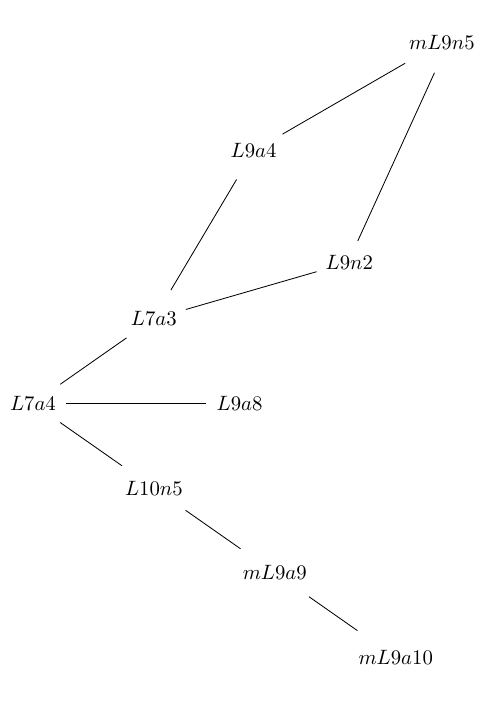}
    \captionof{figure}{($\delta_{1}$,$\delta_{2}$)=(0,$\pm$2)}
    \label{fig:2_pathways}
    \end{center}
\end{minipage}


\begin{table}[h]
    \centering
    \begin{tabular}{|c|c|c|c|c|c|c|c|c|c|c|}
 \hline
 \multicolumn{11}{|c|}{($\delta_{1}$,$\delta_{2}$)=(0,$\pm$1)} \\
 \hline
   & L7n2 & mL8n2 & L9n3 & L7a1 & L9n8 & L9a3 & L8a1 & L9a1 & L9n6 & L9a2 \\
 \hline
 mL5a1 & 1 & 1 & 2 & 2 & 2 & 3* & 2 & 4* & 3 & 5*\\
 \hline
 L7n2 & 0 & 2 & 1 & 3* & 1 & 4* & 1 & 1-5 & 2 & 6* \\
 \hline
 mL8n2 & - & 0 & 3 & 1 & 2 & 2 & 3* & 3* & 4 & 4* \\  
 \hline
 L9n3 & - & - & 0 & 4* & 2 & 5* & 2 & 6* & 1 & 7*\\
 \hline
 L7a1 & - & - & - & 0 & 4* & 1 & 4* & 2 & 5* & 3* \\
 \hline
 L9n8 & - & - & - & - & 0 & 5* & 2 & 6* & 1 & 7* \\
 \hline
 L9a3 & - & - & - & - & - & 0 & 5* & 1 & 6* & 2 \\
 \hline
 L8a1 & - & - & - & - & - & - & 0 & 6* & 3 & 7* \\
 \hline
 L9a1 & - & - & - & - & - & - & - & 0 & 7* & 1\\
 \hline
 L9n6 & - & - & - & - & - & - & - & - & 0 & 8*\\
 \hline
\end{tabular}
    \caption{*Upper bound on the self $\Delta$-Gordian distance.}
    \label{tab:1}
\end{table}
\vspace{7.5mm}

\begin{table}[]
    \centering
    \begin{tabular}{|c|c|c|c|c|c|c|c|c|}
 \hline
 \multicolumn{9}{|c|}{($\delta_{1}$,$\delta_{2}$)=(0,$\pm$2)} \\
 \hline
 Blank & L7a3 & L7a4 & L9a4 & L9a8 & mL9a9 & mL9a10 & L9n2 & mL9n5 \\
 \hline
 L7a3 & 0 & 1 & 1 & 2 & 3* & 4* & 1 & 2   \\
 \hline
 L7a4 & - & 0 & 2 & 1 & 2 & 3* & 2 & 3   \\
 \hline
 L9a4 & - & - & 0 & 3 & 4* & 5* & 2 & 1   \\
 \hline
 L9a8 & - & - & - & 0 & 3* & 4* & 3 & 4   \\
 \hline
 mL9a9 &- &- &- &- &0 &1 &4* &5*  \\
 \hline
 mL9a10 &- &- &- &- &- &0 &5* &6   \\
 \hline
 L9n2 &- &- &- &- &- &- &0 &1   \\
 \hline
 mL9n5 &- &- &- &- &- &- &- &0   \\
 \hline
\end{tabular}
    \caption{*Upper bound on the self $\Delta$-Gordian distance.}
    \label{tab:2}
\end{table}

\section{Delta-Splitting Number}
The splitting number of a link is defined to be the minimal number of crossing changes involving different components needed to deform a link into a split link \cite{batson13}. The splitting number has been tabulated for links with up to 9 crossings \cite{cha13}.

Similarly, as any algebraically split link is $\Delta$-equivalent to the trivial link, we can define the $\Delta$-splitting number $sp^\Delta(L)$ to be the minimum number of $\Delta$-moves needed to deform $L$ into a split link. As the split link one obtains is not necessarily the trivial link, we have $sp^\Delta(L)\leq u^\Delta(L)$. The mixed $\Delta$-splitting number $sp^{m\Delta}(L)$ is defined to be the minimum number of mixed $\Delta$-moves needed to deform $L$ into a split link. The following theorem and its corollary shows that $sp^{m\Delta}(L)$ is also well-defined for algebraically split links.





\begin{theorem}
    Two links are mixed $\Delta$-equivalent if and only if they have the same pairwise linking numbers and components.
\end{theorem}

\begin{proof}
    One direction follows immediately from the fact that $\Delta$-moves preserve linking number and does not change knot type of the components. For the converse, suppose $m$-component links $L$ and $L'$ have the same pairwise linking numbers and equal components $L_i=L_i'$ for $1\leq i\leq m$. Let $\lambda_{ij}$ denote $lk(L_i,L_j)=lk(L_i',L_j')$ for all $1\leq i<j\leq m$ up to renumbering components. There exists a link $S$ with the same linking information as $L$ comprising the same $m$ components $S_i=L_i$ such that any two components $S_i$ and $S_j$ intersect in $\lambda_{ij}$ clasps as in Figure \ref{fig:clasp_link_S}. We will argue that $L$ is mixed $\Delta$-equivalent to $S$, thus $L'$ is also mixed $\Delta$-equivalent to $S$, and hence $L$ and $L'$ are mixed $\Delta$-equivalent to one another.

    \begin{figure}[h]
     \includegraphics[height=2.4in]{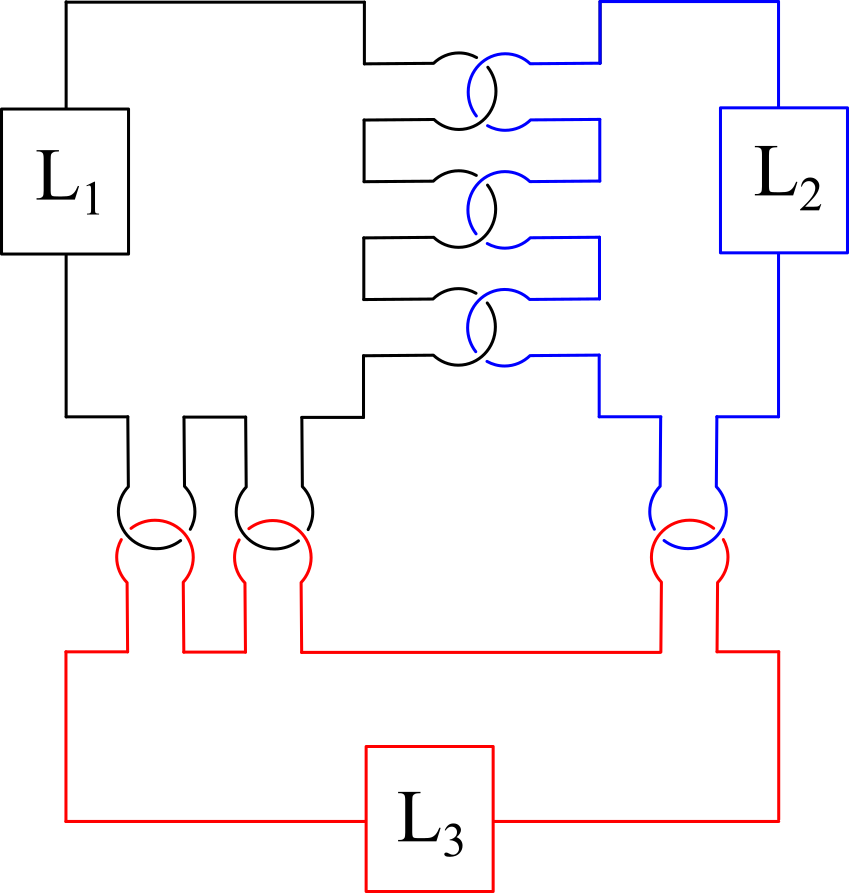}
    \caption{A link $S$ with linking numbers $\lambda_{12}=3$, $\lambda_{13}=2$, and $\lambda_{23}=-1$.}
    \label{fig:clasp_link_S}
\end{figure}

    First observe that given the components $L_i$ and $L_j$ in $L$, we can deform them so that the components intersect in pairs of clasps such that the difference between clasps contributing +1 as those clasps contributing -1 to the linking number is $\lambda_{ij}$. If there exists both positive and negative clasps, we can arrange the diagram such that two clasps with opposite contributions to linking number are adjacent to each other, possibly with other clasps representing intersections of the component with itself or other components between the claps, as in Figure \ref{fig:clasp_link}.a.

    \begin{figure}[h]
     \includegraphics[height=2.4in]{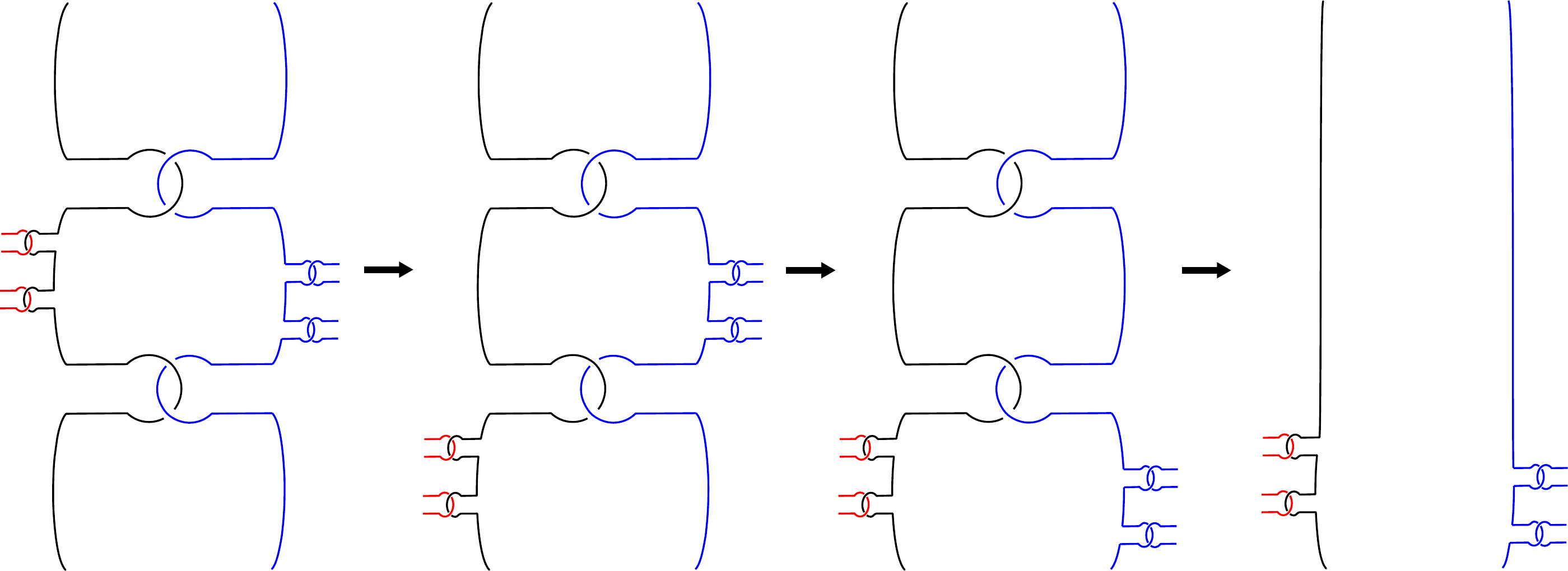}
    \caption{Separating two components of an algebraically split link with mixed-$\Delta$ moves.}
    \label{fig:clasp_link}
\end{figure}

\begin{figure}[h]
    \includegraphics[height=1.1in]{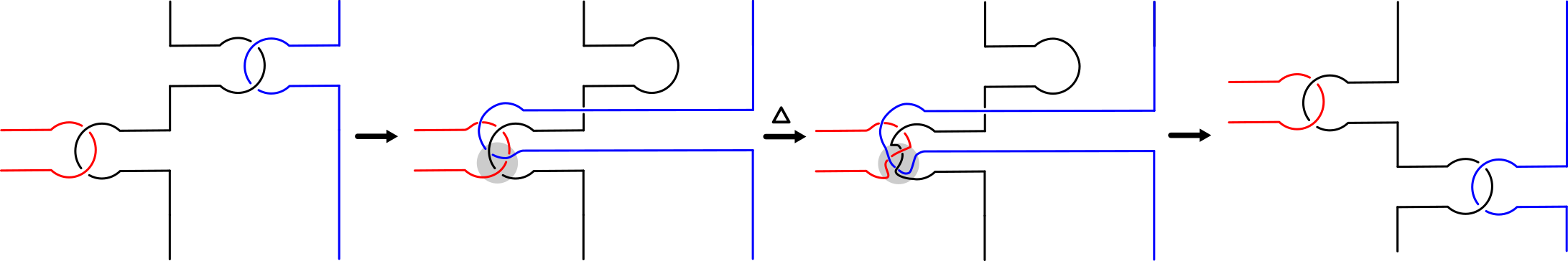}
    \caption{Moving a clasp over another with a mixed-$\Delta$ move.}
    \label{fig:clasp_move}
\end{figure}

    We can move a clasp between different components over another clasp using a single mixed-$\Delta$ move, as in Figure \ref{fig:clasp_move}. Repeated application of this move allows us to move any pairs of positive and negative clasps adjacent to each other with no other clasps between them as in Figure \ref{fig:clasp_link}.b and \ref{fig:clasp_link}.c. The two clasps then cancel, as in Figure \ref{fig:clasp_link}.d.

    We can repeat this process until all clasps between $L_i$ or $L_j$ are negative or positive. But then there are exactly $|\lambda_{ij}|$ such clasps, the sign of which is determined by the sign of $\lambda_{ij}$. Hence, we have transformed $L$ into $S$ via mixed $\Delta$-moves.
\end{proof}

Recall that an algebraically split link is a link with pairwise vanishing splitting number. Thus it follows:
\begin{corollary}
    A link $L$ is mixed $\Delta$-equivalent to a split link if and only if $L$ is an algebraically split link.
\end{corollary}

Hence the mixed $\Delta$-splitting number, denoted $sp^{m\Delta}(L)$, is well-defined for an algebraically split link $L$. And since mixed $\Delta$-moves are a subset of $\Delta$-moves more generally (self or mixed), we have $sp^{m\Delta}(L)\geq sp^{\Delta}(L)$. And as a mixed $\Delta$-move does not change the knot type of any of the components of the link, the split link obtained by mixed $\Delta$-moves on $L$ will be the disjoint union of the components of $L$.\\

\begin{proposition}
Given an algebraically split link $L$, \[sp^{m\Delta}(L)\equiv \arf(L)+\sum_{i=1}^m \arf(L_i)\pmod{2}.\] 
\end{proposition}
\begin{proof}
    Let $S$ denote the split link comprising the disjoint union of the components of $L$. Then,
    \[sp^{m\Delta}(L)=d^{m\Delta}_G(L,S)\equiv \arf(L)+\arf(S) \equiv \arf(L)+\sum_{i=1}^m \arf(L_i)\pmod{2}.\]
\end{proof}

\begin{example}

        \begin{figure}
            \centering
        \begin{tikzpicture}
            \node[inner sep=0pt] (Phase1) at (0,0) {\includegraphics[width=.25\textwidth]{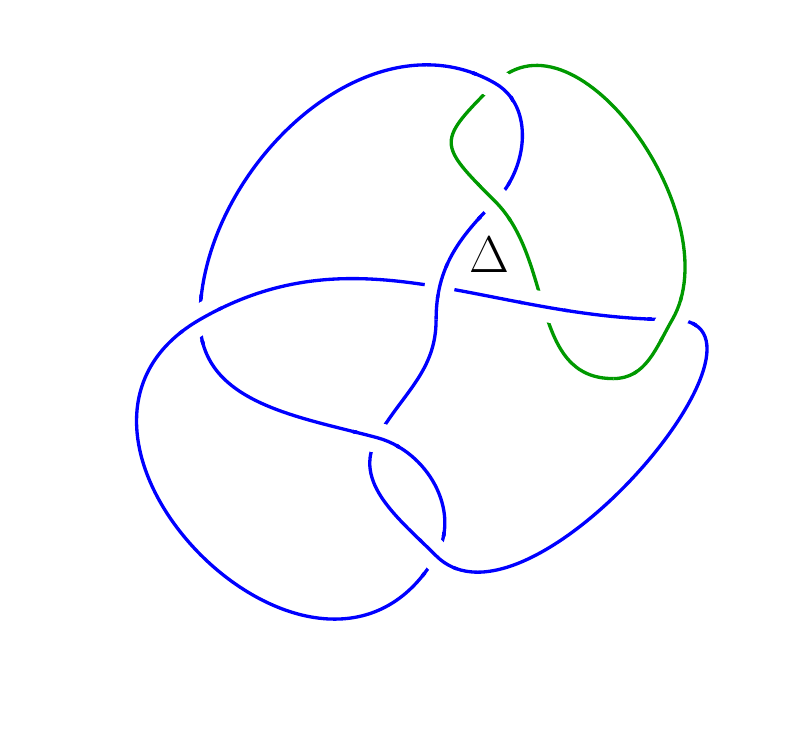}};
            
            \node[inner sep=0pt] (Phase2) at (5,0) {\includegraphics[width=.25\textwidth]{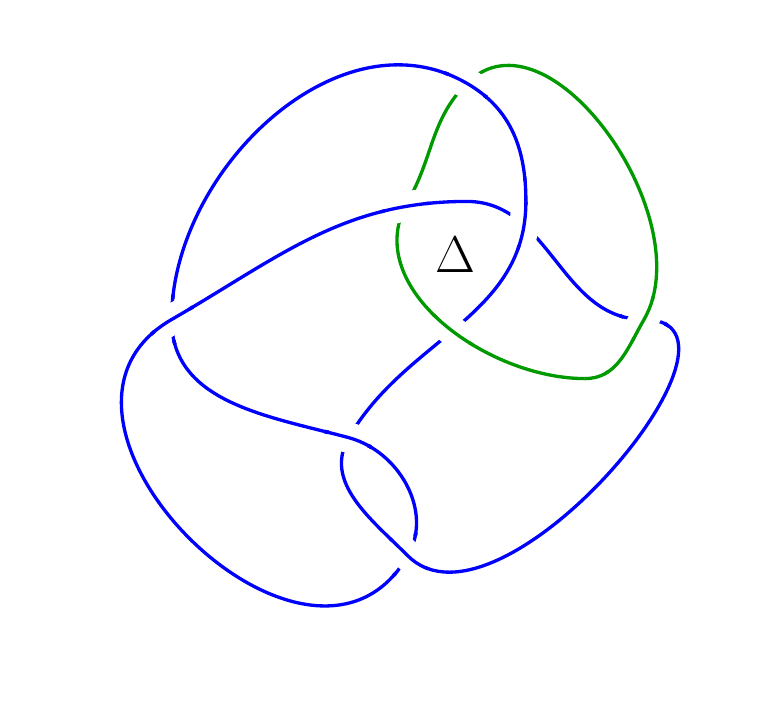}};
            
            \node[inner sep=0pt] (Phase3) at (10,0) {\includegraphics[width=.25\textwidth]{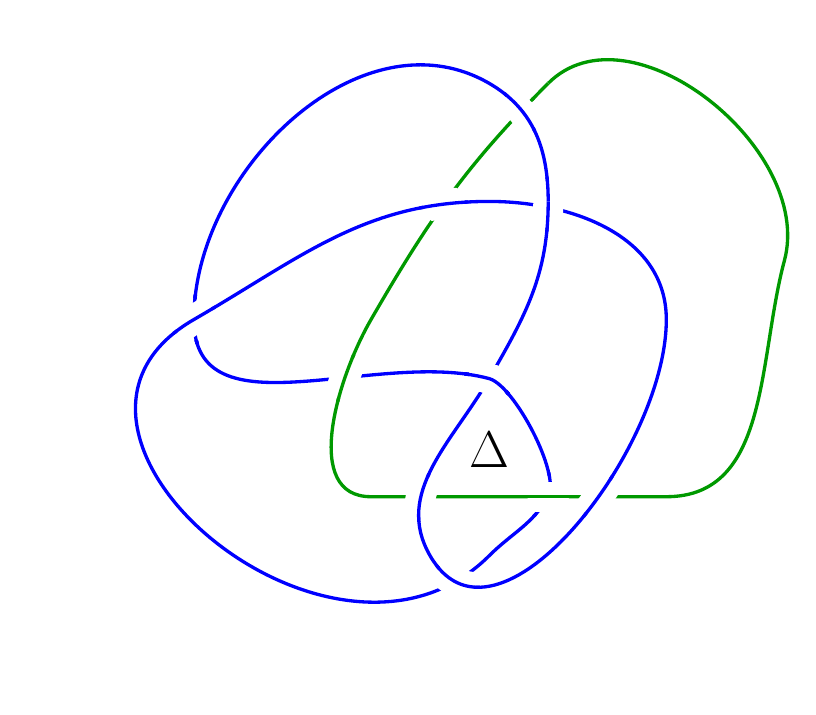}};
            
            \node[inner sep=0pt] (Phase4) at (7,-3) {\includegraphics[width=.25\textwidth]{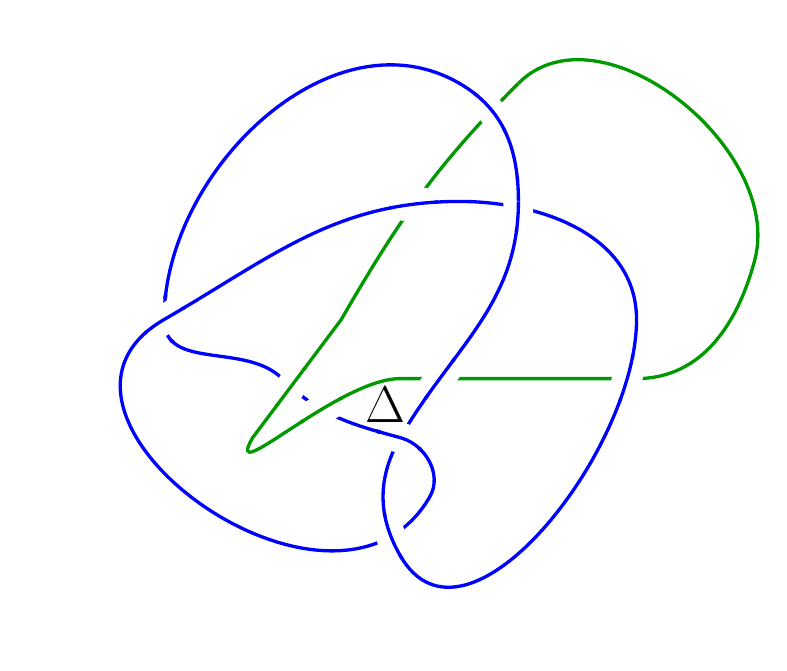}};
            
            \node[inner sep=0pt] (Phase5) at (2.25,-3) {\includegraphics[width=.30\textwidth]{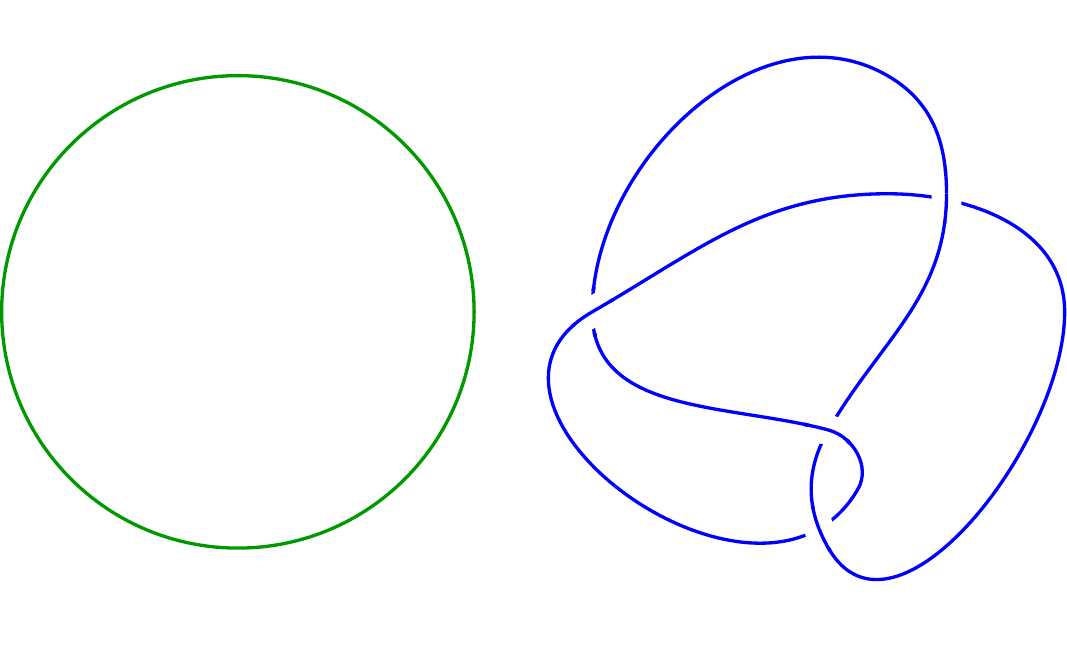}};
            
            \draw[<->,thick] (Phase1.east) -- (Phase2.west);
            \node at (2.5,0.25) {$\Delta$};
            \node[rotate=50] at (9.3,-2.3) {$\Delta$};
            \node at (7.5,0) {$\approx$};
            \node at (5,-2.55) {$\approx$};
            \draw[<->,thick] (Phase3.south) -- (Phase4.east);
        \end{tikzpicture}
        \caption{$L8a2$ is two mixed $\Delta$-moves away from a split link.}
        \label{fig:splitL8a2}
        \end{figure}
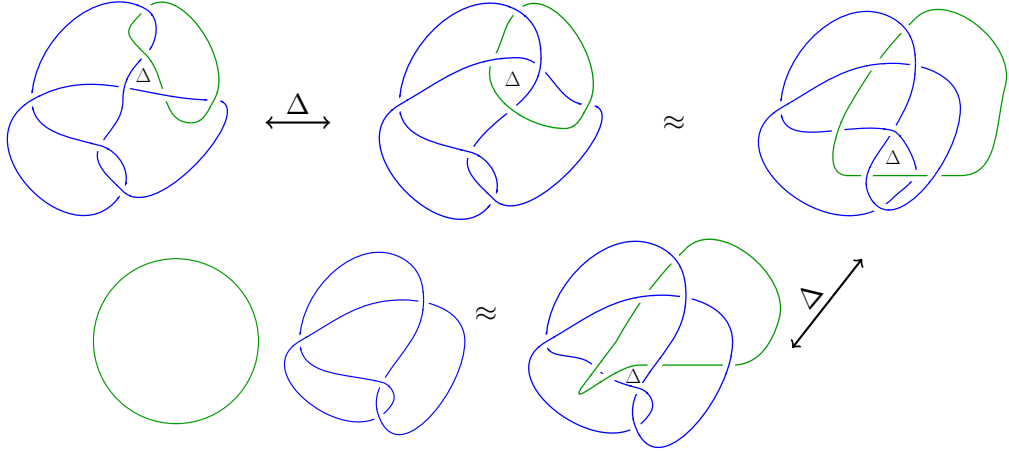
    
    The link $L8a4$ comprises an unknot and trefoil. There exists a length two mixed $\Delta$-pathway to the split link: $L8a2, mL8n2, 0_1\sqcup 4_1$. Therefore $1\leq sp^{m\Delta}(L8a4)\leq 2$. However, as
    \[sp^{m\Delta}(L8a4)\equiv \arf(L8a4)+\left(\arf(0_1)+\arf(m3_1)\right)= 1+(0+1)\equiv 0\pmod{2}\]
    we have $sp^{m\Delta}(L8a4)=2$. But as $L8a4$ can be deformed into the trivial link with a single self $\Delta$-move on the trefoil, we have $sp^{\Delta}(L8a4)=1$.
\end{example}

Moreover, we have the following relationship with the $\Delta$-unknotting and $\Delta$-unlinking number.\\

\begin{proposition}
    Given an algebraically split link $L=L_1\sqcup\cdots\sqcup L_m$, we have:
    \[sp^{m\Delta}(L)\geq u^{\Delta}(L)-\sum_{i=1}^m u^{\Delta}(L_i).\]
\end{proposition}
\begin{proof}
    Let $S$ denote the link comprising the split union of $L_1,...,L_m$. Then $L$ can be deformed into $S$ with $sp^{m\Delta}(L)$ mixed $\Delta$-moves and each component of $S$ can be deformed into the unknot in $u^{\Delta}(L_i)$ $\Delta$-moves. Since mixed $\Delta$-moves preserve the knot type of the components and $u^\Delta(L)$ is the minimal number of $\Delta$-moves needed to deform $L$ into the trivial link, we have \[u^{\Delta}(L) \leq sp^{m\Delta}(L)+\sum_{i=1}^m u^{\Delta}(L_i).\]
\end{proof}

\begin{example}
    The link $L9a18$ has two unknotted components and the unknotting number for $L9a18$ is three [Bos+]. Therefore, $sp^{m\Delta}(L9a18)\geq 3$. Moreover, there exists a mixed $\Delta$-pathway of length three from $L9a18$ to $L7a4$ to $L5a1$ to the trivial link. Thus $sp^{m\Delta}(L9a18)=3$.
\end{example}

Finally, a mixed $\Delta$-move can be achieved with two crossing changes, each between different components, as in Figure \ref{fig:mixed_delta_crossing}. Thus $2sp^{m\Delta}(L)\geq sp(L)$ or equivalently:

\begin{proposition}
    $sp^{m\Delta}(L)\geq \frac{1}{2}sp(L).$
\end{proposition}

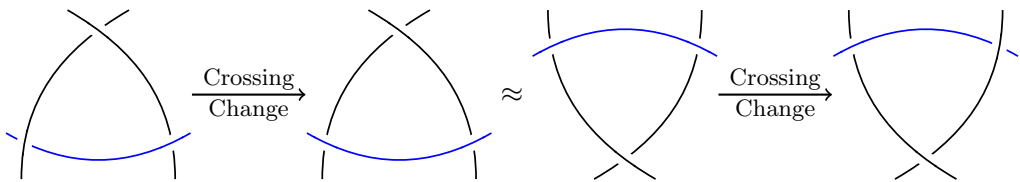
\begin{figure}[h]
    \begin{tikzpicture}
    \node[inner sep=0pt] (cc1) at (0,0)
        {\includegraphics[width=.15\textwidth]{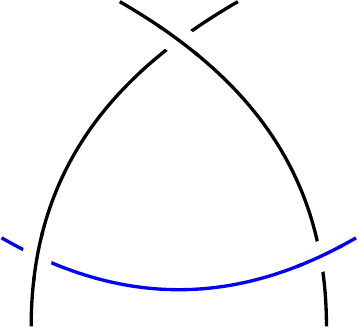}};
    \node[inner sep=0pt] (cc2) at (4,0)
        {\includegraphics[width=.15\textwidth]{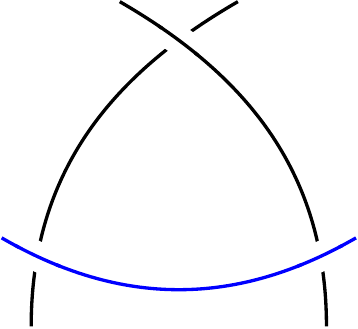}};
    \node[inner sep=0pt] (cc3) at (7,0)
        {\includegraphics[width=.15\textwidth]{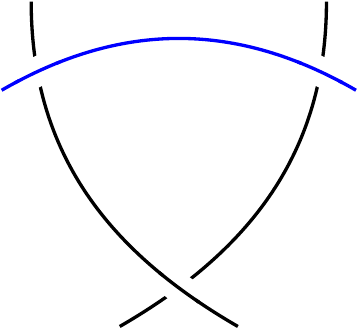}};
    \node[inner sep=0pt] (cc4) at (11,0)
        {\includegraphics[width=.15\textwidth]{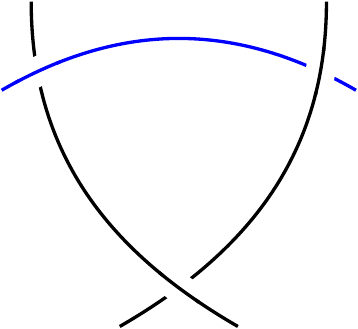}};
    
    \draw[->,thick] (cc1.east) -- (cc2.west);
    \node at (2,0.2) {\footnotesize{Crossing}};
    \node at (2,-0.2) {\footnotesize{Change}};
    \node at (5.5,0) {$\approx$};
    \draw[->,thick] (cc3.east) -- (cc4.west);
    \node at (9,0.2) {\footnotesize{Crossing}};
    \node at (9,-0.2) {\footnotesize{Change}};
    \end{tikzpicture}
    \caption{A mixed $\Delta$-move as two crossing changes.}
    \label{fig:mixed_delta_crossing}
\end{figure}

Together, these bounds help us determine the $\Delta$-splitting number and mixed $\Delta$-splitting number for algebraically split links with up to 9 crossings.

    \begin{figure}
        \centering
        \includegraphics[height=3in]{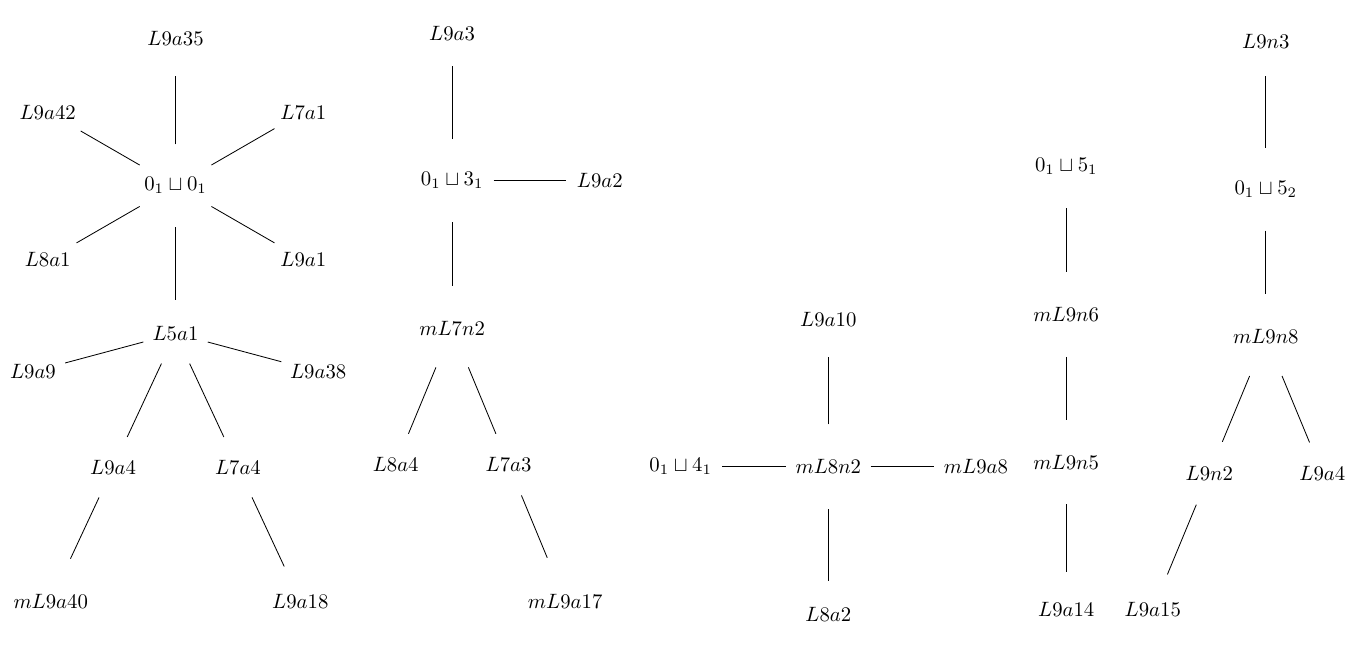}
        \caption{Edges represent a mixed $\Delta$-move.}
    \label{fig:mixed_delta_graph}
    \end{figure}

    \begin{table}
    \centering
    \scriptsize{
    \begin{tabular}{|c|c||c||c||c|c|c|c|c|} 
        \hline
        Link & Components & $sp^\Delta(L)$ & $sp^{m\Delta}(L)$ & $u^\Delta(L)$ & $\Sigma u^\Delta(L_i)$ & $sp(L)$ & $\text{Arf}(L)$ & $\Sigma \text{Arf}(L_i)$\\
        \hline
        L5a1 & $0_1,0_1$ & 1 & 1 & 1 & 0 & 2 & 1 & 0\\
        \hline
        L6a4 & $0_1,0_1,0_1$ & 1 & 1 & 1 & 0 & 2 & 1 & 0\\
        \hline
        L7a1 & $0_1,0_1$ & 1 & 1 & 1 & 0 & 2 & 1 & 0\\
        \hline
        L7a3 & $0_1,3_1$ & 2 & 2 & 3 & 1 & 2 & 1 & 1\\
        \hline
        L7a4 & $0_1,0_1$ & 2 & 2 & 2 & 0 & 2 & 0 & 0\\
        \hline
        L7n2 & $0_1,m3_1$ & 1 & 1 & 2 & 1 & 2 & 0 & 1\\
        \hline
        L8a1 & $0_1,0_1$ & 1 & 1 & 1 & 0 & 2 & 1 & 0\\
        \hline
        L8a2 & $0_1,4_1$ & 1 & 2 & 1 & 1 & 2 & 1 & 1\\
        \hline
        L8a4 & $0_1,m3_1$ & 1 & 2 & 1 & 1 & 2 & 1 & 1\\
        \hline
        L8n2 & $0_1,4_1$ & 1 & 1 & 2 & 1 & 2 & 0 & 1\\
        \hline
        L9a1 & $0_1,0_1$ & 1 & 1 & 1 & 0 & 2 & 1 & 0\\
        \hline
        L9a2 & $0_1,3_1$ & 1 & 1 & 2 & 1 & 2 & 0 & 1\\
        \hline
        L9a3 & $0_1,3_1$ & 1 & 1 & 2 & 1 & 2 & 0 & 1\\
        \hline
        L9a4 & $0_1,5_2$ & 2 & 2 & 4 & 2 & 2 & 0 & 0\\
        \hline
        L9a8 & $0_1,4_1$ & 2 & 2 & 3 & 1 & 2 & 1 & 1\\
        \hline
        L9a9 & $0_1,0_1$ & 2 & 2 & 2 & 0 & 2 & 0 & 0\\
        \hline
        L9a10 & $0_1,4_1$ & 2 & 2 & 3 & 1 & 2 & 1 & 1\\
        \hline
        L9a14 & $0_1,5_1$ & 1 - 3 & 1 or 3 & 4 or 6 & 3 & 2 & 0 & 1\\
        \hline
        L9a15 & $0_1,5_2$ & 1 - 3 & 1 or 3 & 3 or 5 & 2 & 2 & 1 & 0\\
        \hline
        L9a17 & $0_1,m3_1$ & 1 - 3 & 1 or 3 & 2 or 4 & 1 & 2 & 0 & 1\\
        \hline
        L9a18 & $0_1,0_1$ & 2 or 3 & 3 & 3 & 0 & 2 & 1 & 0\\
        \hline
        L9a35 & $0_1,0_1$ & 1 & 1 & 1 & 0 & 2 & 1 & 0\\
        \hline
        L9a38 & $0_1,0_1$ & 2 & 2 & 2 & 0 & 2 & 0 & 0\\
        \hline
        L9a40 & $0_1,0_1$ & 2 or 3 & 3 & 3 & 0 & 4 & 1 & 0\\
        \hline
        L9a42 & $0_1,0_1$ & 1 & 1 & 1 & 0 & 2 & 1 & 0\\
        \hline
        L9a53 & $0_1,0_1,0_1$ & 1 & 1 & 1 & 0 & 2 & 1 & 0\\
        \hline
        L9a54 & $0_1,0_1,0_1$ & 2 or 3 & 3 & 3 & 0 & 4 & 1 & 0\\
        \hline
        L9n2 & $0_1,5_2$ & 2 & 2 & 4 & 2 & 2 & 0 & 0\\
        \hline
        L9n3 & $0_1,m5_2$ & 1 & 1 & 3 & 2 & 2 & 1 & 0\\
        \hline
        L9n5 & $0_1,m5_1$ & 2 & 2 & 5 & 3 & 2 & 1 & 1\\
        \hline
        L9n6 & $0_1,m5_1$ & 1 & 1 & 4 & 3 & 2 & 0 & 1\\
        \hline
        L9n8 & $0_1,m5_2$ & 1 & 1 & 3 & 2 & 2 & 1 & 0\\
        \hline
        L9n25 & $0_1,0_1,0_1$ & 2 & 2 & 2 & 0 & 2 & 0 & 0\\
        \hline
        L9n27 & $0_1,0_1,0_1$ & 2 & 2 & 2 & 2 & 4 & 0 & 0\\
        \hline
    \end{tabular}
    }
    \caption{Determine (mixed) $\Delta$-splitting number.}
    \label{tab:mixed_delta}
    \end{table}

\clearpage
\printbibliography

\end{document}